\documentclass{amsart}
\usepackage[nohug]{diagrams}
\usepackage[mathscr]{eucal}

\usepackage{amsfonts}
\usepackage{amsmath}
\usepackage{amsthm}
\usepackage{amssymb}
\usepackage{latexsym}

\newtheorem{theorem}[]{Theorem}
\newtheorem{proposition}[]{Proposition}

\newtheorem{lemma}[]{Lemma}
\theoremstyle{definition}

\newtheorem{remark}[]{Remark}

\newcommand{\beq}{\begin{equation}\label}

\newcommand{\id}{{{\mathtt {Id}}}}
\newcommand{\im}{{{\mathtt {Im}}}}

\newcommand{\into}{\,\,\hookrightarrow\,\,}

\newcommand{\onto}{\,\,\twoheadrightarrow\,\,}

\newcommand{\Rep}{{\mathtt{Rep}}}
\newcommand{\Irr}{{\mathtt{Irr}}}

\newcommand{\End}{{\mathtt{End}}}

\newcommand{\Hom}{{\mathtt{Hom}}}

\newcommand{\Tor}{{\mathtt{Tor}}}

\newcommand{\Mod}{{\mathtt{Mod}}}

\newcommand{\Ker}{{\mathtt{Ker}}}

\newcommand{\Ann}{{\mathtt{Ann}}}
\renewcommand{\dim}{{\mathtt{dim}}}

\newcommand{\rk}{{\mathtt{Rk}}}

\newcommand{\tr}{{{\mathtt{Tr}}}}

\newcommand{\GL}{\mathtt{GL}}
\newcommand{\G}{\mathtt{G}}
\newcommand{\SL}{\mathtt{SL}}
\newcommand{\LL}{\mathtt{L}}
\newcommand{\SSp}{\mathtt{Sp}}

\newcommand{\e}{{\mathbf{e}}}

\newcommand{\RR}{{\mathcal{R}}}
\newcommand{\AB}{{\mathcal{A}}}

\newcommand{\bGamma}{\boldsymbol{\Gamma}}

\newcommand{\ee}{{\mathfrak{e}}}

\newcommand{\A}{{\mathsf{A}}_{\infty}}

\newcommand{\la}{\label}

\newcommand{\bxi}{\boldsymbol{\xi}}

\def\C{{\mathbb{C}}}
\def\V{{\bar{V}}}
\def\X{{\bar{X}}}
\def\Y{{\bar{Y}}}
\def\w{{\bar{w}}}
\def\v{{\bar{v}}}
\def\e{{\bar{e}}}

\def\R{{\mathcal{R}}}

\def\Z{{\mathbb{Z}}}

\def\bn{\boldsymbol{n}}
\def\N{\mathcal{N}}
\def\NN{\mathbb{N}}

\def\O{{\mathcal O}}

\def\LL{{\mathbf{L}}}

\def\tpi{\tilde{\pi}}
\def\g{{\mathfrak{g}}}
\def\h{{\mathfrak{h}}}
\def\b{{\mathfrak{b}}}

\def\sll2{{\mathfrak{s}\mathfrak{l}}_2}

\def\CC{{\mathcal{C}}}

\def\DD{\Delta}

\def\S{{\mathcal S}}

\begin{document}
\title[Recollement of Deformed Preprojective Algebras]{Recollement
of Deformed Preprojective Algebras and the Calogero-Moser
Correspondence}
\author{Yuri Berest}
\address{Department of Mathematics,
Cornell University, Ithaca, NY 14853-4201, USA}
\email{berest@math.cornell.edu}
\thanks{Research partially supported by NSF grant DMS 04-07502.}
\author{Oleg Chalykh}
\address{Department of Mathematics, University of Leeds, Leeds LS2 9JT, UK}
\email{oleg@maths.leeds.ac.uk}
\author{Farkhod Eshmatov}
\address{Department of Mathematics,
Cornell University, Ithaca, NY 14853-4201, USA}
\email{eshmatov@math.cornell.edu} \maketitle

\vspace{-0.2in}

\begin{center}
\textit{To Vitya Ginzburg for his 50th birthday}
\end{center}

\vspace{0.1in}

\section{Introduction}

The aim of this paper is to clarify the relation between the
following objects: $\,(a)\,$ rank $1$ projective modules (ideals)
over the first Weyl algebra $ A_1(\C)$; $\,(b)\,$ simple modules
over deformed preprojective algebras $\,\Pi_{\lambda}(Q)\,$
introduced by Crawley-Boevey and Holland \cite{CBH}; and $\, (c)
\,$ simple modules over the rational Cherednik algebras
$\,H_{0,c}(S_n)\,$ associated to symmetric groups \cite{EG}. The
isomorphism classes of each type of these objects can be 
parametrized naturally by the same space (namely, the Calogero-Moser 
algebraic varieties); however, no natural functors
between the corresponding module categories seem to be known. We
construct such functors by translating our earlier results on
$\A$-modules over $ A_1 $ (see \cite{BC}) to a more familiar
setting of representation theory. 

In the last section we extend our construction to the case of Kleinian 
singularities $\, \C^2\!/\Gamma \,$, where $ \Gamma $ is a 
finite cyclic subgroup of $ \SL(2, \C) $. In this case, the Weyl algebra 
is replaced by a ``quantized coordinate ring'' $\, O_{\tau}(\Gamma)\,$ 
of $\, \C^2\!/\Gamma \,$, the Calogero-Moser spaces by certain quiver
varieties $ \CC_{\bn, \tau}(\Gamma) $, and the rational Cherednik algebra by 
a symplectic reflection algebra $ H_{0,k,c}(\bGamma_n) $ associated
to the $n$-th wreath product $\, \boldsymbol{\Gamma}_n = S_n \ltimes
\Gamma^n\,$ 
(see \cite{BGK1, BGK2}).

We should mention that the question of explaining the ``mysterious''
bijection between the ideal classes of $ A_{1} $ and simple modules 
of $\,H_{0,c}(S_n)\,$ (and more generally, a similar bijection for the algebras 
$\, O_{\tau}(\Gamma)\,$ and $ H_{0,k,c}(\bGamma_n)
$) was first raised in \cite{EG} and emphasized further in \cite{BGK2} 
(see Remark~1.1 in {\it loc.\!\! cit.}).

In the end, we would like to thank W.~Crawley-Boevey, I.~Gordon and
G.~Wilson for interesting questions and comments. We are also grateful
to V.~Ginzburg for explaining to us some of the results of \cite{EGGO}. 
The results of this paper were announced at a conference on Interactions
between Noncommutative Algebra and Geometry in
Oberwolfach in May 2006 (see \cite{B}). The first author takes this
opportunity to thank the organizers T.~Stafford and M.~van den Bergh for
inviting him to speak at this conference.

\section{The Calogero-Moser Correspondence}

\subsection{The Calogero-Moser spaces}
\la{CMS} For an integer $\, n \geq 0 \,$, let $ \tilde{\CC}_n $ be
the space of linear maps
$$
\{(\X, \Y, \v, \w) \,:\, \X, \Y \in \End(\C^n) \, , \, \v \in
\Hom(\C,\C^n) , \, \w \in \Hom(\C^n, \C)\} \ ,
$$
satisfying the equation $\, [\X,\,\Y] + \id + \v\,\w = 0 \,$. The
group $\, \GL(n, \C) \,$ acts on $\, \tilde{\CC}_n \,$ in the
natural way:
$$
(\X, \Y, \v, \w) \mapsto (g\X g^{-1}, \,g\Y g^{-1},\, g \v,\,\w
g^{-1}) , \ g \in \GL(n, \C)\ ,
$$
and we define the $n$-th {\it Calogero-Moser space} $\, \CC_n \,$
to be the quotient variety $\, \tilde{\CC}_n/\!/\GL(n,\C) \,$
modulo this action. It turns out that $\, \GL(n, \C) \,$ acts
freely on $\, \tilde{\CC}_n \,$, and $ \CC_n $ is a smooth affine
variety of dimension $ 2n $ (see \cite{W}).

\subsection{Deformed preprojective algebras}
\la{DPA} Let $ Q = (I, \, Q) $ be a finite quiver with vertex set
$ I $ and arrow set $ Q $, and let $\, \bar{Q}\,$ be its double
(i.e. the quiver obtained from $ Q $ by adding a reverse arrow $
a^* $ to each arrow $ a \in Q $). Following \cite{CBH}, for each
$\,\lambda = (\lambda_i) \in \C^{I}\,$, we define the {\it
deformed preprojective algebra of weight $ \lambda $} by
$$
\Pi_{\lambda}(Q) := \C\bar{Q}\left/\left\langle\sum_{a \in Q}\,
[a,\,a^*] - \sum_{i \in I}\, \lambda_i e_i \right\rangle\right.\ .
$$
Here $ \C\bar{Q} $ denotes the path algebra of the double quiver $
\bar{Q} $ and $\, e_i \in \C\bar{Q} \,$ stand for the orthogonal
idempotents corresponding to the trivial paths (vertices) in $
\bar{Q} $.

In this section we will be concerned with the following example. 
Let $ Q_\infty $ be the
quiver consisting of two vertices $\, \{\infty, \, 0\} \,$ and
two arrows $ v:\, 0 \to \infty $ and $ X:\, 0 \to 0 $. Write
$\, w := v^* \,$ and $\,
Y := X^* \,$ for the reverse arrows in $ \bar{Q}_\infty $. The
algebra $ \Pi_\lambda := \Pi_{\lambda}(Q_\infty) $ is then
generated by $\, X, Y, v, w\,$ and the idempotents $ e_0 $ and $
e_\infty $, which, apart from the standard path algebra relations,
satisfy
\begin{equation}
\la{E25} [X,\,Y] - wv = \lambda_0 e_0 \ ,\quad vw = \lambda_\infty
e_\infty \ .
\end{equation}
Thus, right $ \Pi_{\lambda}$-modules can be identified with
representations $ V = V_\infty \oplus V_0 $ of $ \bar{Q}_\infty $,
in which linear maps $ \,\X, \Y \in \End(V_0), \,\v \in
\Hom(V_\infty, V_0),\, \w \in \Hom(V_0, V_\infty) \,$,
corresponding to the (right) action of $\, X, Y, v, w \,$, satisfy
the equations
\begin{equation}
\la{E26} [\X,\Y] + \v\w = - \lambda_0 \, \id_{V_0} \ , \quad \w\v
= \lambda_\infty \, \id_{V_\infty} \ .
\end{equation}

\subsection{Representation varieties}
\la{RV} Set $ \lambda_0 = 1 $ and $ \lambda_\infty = -n $ in the
above example. Comparing \eqref{E26} to the definition of the
Calogero-Moser spaces in Section~\ref{CMS}, we see that each point
of $ \tilde{\CC}_n $ corresponds naturally to a right $
\Pi_{\lambda} $-module of dimension vector $ \alpha = (1,n) $. All
such modules are simple and, as it was originally observed by
Crawley-Boevey (see \cite{CB2} or Proposition~\ref{exist} below),
every semisimple module of $ \Pi_{\lambda} $ of dimension $ \alpha $
has this form. Thus, we can identify the
varieties $ \CC_n $ with representation spaces $\,
\Rep(\Pi_{\lambda}^{\rm \footnotesize opp},\, \alpha)/\!/\G(\alpha)\,$,
parametrizing the isomorphism classes of simple modules of
dimension vector $ \alpha $.

On the other other hand, according to \cite{BW, BW1}, the
Calogero-Moser varieties $ \CC_n $ also parametrize the
isomorphism classes of right ideals of the first Weyl algebra $
A_1(\C) $. Our aim is to relate simple modules of $ \Pi_{\lambda}
$ to ideals of $ A_1 $ in a natural (functorial) way. To this end
we will use a version of ``recollement'' formalism, due to
Beilinson, Bernstein and Deligne \cite{BBD} (see also \cite{CPS}).

\subsection{Recollement}
\la{R} Originating from geometry, this formalism abstracts
functorial relations between the categories of abelian sheaves on
a topological space $ X $, its open subspace $U$ and the closed
complement of $ U $ in $ X $. Unlike \cite{BBD}, we will work with
abelian (rather than triangulated) categories. For reader's
convenience, we review the definition of recollement in this
special case.

Let $ \AB $, $ \AB' $ and $ \AB'' $ be three abelian categories.
We say that $ \AB $ is a {\it recollement} of  $ \AB' $ and $
\AB'' $ if there are six additive functors
\begin{equation}
\la{D1}
\begin{diagram}[small, tight]
\AB' 
& \ \pile{\lTo^{i^*}\\ \rTo^{i_*} \\ \lTo^{i^!}} \ & 
\AB
& \ \pile{\lTo^{j_!}\\ \rTo^{j^*}\\ \lTo^{j_*}} \ & 
\AB'' 
\end{diagram}
\end{equation}
satisfying the following conditions:

(R1)  \ $ i^* $ and $ i^{!} $ are adjoint to $ i_* $ on the left
and on the right respectively, the adjunction morphisms $\,
\id_{\AB'} \to i^{!} i_* \,$ and $\, i^*i_* \to \id_{\AB'} \,$
being isomorphisms;

(R2) \ $ j_{!} $ and $ j_* $ are adjoint to $ j^* $ on the left
and on the right respectively, the adjunction morphisms $\,
\id_{\AB''} \to j^*j_! \,$ and $\, j^*j_* \to \id_{\AB''} \,$
being isomorphisms;

(R3) \  $ i_* $ is an embedding of $ \AB' $ onto the full
subcategory of $ \AB $ consisting of objects annihilated by $ j^*
$; thus $ j^* $ induces an equivalence of abelian categories $\,
\AB/\AB' \simeq \AB'' \,$.

These conditions imply

(R4)  \  $\, i^* j_! = 0\,$ and $\, i^! j_* = 0 \,$ (by adjunction
with $\, j^* i_* = 0 \,$);

(R5)  \ the canonical morphisms $\,\sigma: j_! j^* \to \id_{\AB}
\,$ and $\,\eta: \id_{\AB} \to j_*j^* \,$ give rise to the exact
sequences of natural transformations
\begin{equation}
\la{seqs} j_{!} j^* \to  \id_{\AB} \to i_* i^* \to 0\quad , \quad
0 \to i_{*} i^{!} \to \id_{\AB} \to j_* j^* \ .
\end{equation}

We will need some simple consequences of the above definition,
which we record in the next two lemmas. 
\begin{lemma}
\la{Lexact} Let $ L $ be an object of $ \AB $ such that $\, i^* L
= i^! L = 0 \,$. Then

$(a)\ \sigma_L:  j_{!} j^* L \to L \,$ is an epimorphism in 
$ \AB $ with $\, \Ker(\sigma_L) \cong i_* i^{!} j_! j^* L  \,$.

$(b)\,$ If $ \AB $ has enough projectives, then we also have
$\,\Ker(\sigma_L) \cong i_* (\LL_1 i^*) L \,$, where $\, \LL_1
i^*:\, \AB \to \AB' \,$ is the left derived functor of $ i^* $.
\end{lemma}
\begin{proof}
$(a)$ $\,\sigma_L $ being epimorphism follows immediately from the
first exact sequence of \eqref{seqs}. To compute its kernel we
apply $ j^* $ to the exact sequence $\, 0 \to
\Ker(\sigma_L) \to j_{!} j^* L \to L \to 0 \,$. By (R2), $\, j^*
\,$ maps $ \sigma_L $ to an isomorphism in $ \AB'' $; hence $\,
j^*(\Ker \, \sigma_L) = 0 \,$ and therefore $\, \Ker(\sigma_L) =
i_* K \,$ for some $ K \in {\rm Ob}(\AB') $. Now, to see that $ K
\cong i^{!} j_! j^* L  $ we apply the (left exact) functor $\, i^!
\,$ to $\, 0 \to i_* K \to j_{!} j^* L \to L \to 0 \,$
and use the first isomorphism of (R1).

$(b)$ If the category $ \AB $ has enough projectives, so does its
quotient $ \AB''$. Moreover, being left adjoint to an exact functor,
 $\, j_!: \AB'' \to \AB \,$  maps projectives to 
projectives. Under these conditions, there is a Grothendieck spectral 
sequence with $\, E^2_{p,q} = (\LL_p i^*) (\LL_q j_!) \,$ converging 
to $\, H_n = \LL_{n}(i^* j_!) \,$. In view of (R4), the limit 
terms of this spectral sequence are zero, and hence so is its 
edge map $\,  E^2_{1,0} \cong  E^\infty_{1,0} \into H_1 \,$. This yields
$\, (\LL_1 i^*)j_! = 0 \,$. Now, applying $\, i^* $ 
to the exact sequence $\, 0 \to i_* K \to j_{!} j^* L \to L \to 0 \,$ 
and using again (R4), we conclude
$\, \Ker(\sigma_L) = i_* K \cong  i_* (\LL_1 i^*) L \, $.
\end{proof}

Next, in addition to the six basic functors \eqref{D1}, we introduce
another additive functor $\, j_{! *}:\, \AB'' \to \AB \,$ (cf. \cite{BBD},
Section~1.4.6). By axiom (R2), for any $\, X,\, Y \in {\rm
Ob}(\AB'')\,$, there are natural isomorphisms
\begin{equation}
\la{norm} \Hom_{\AB}(j_{!} X,\, j_{*}Y) \cong \Hom_{\AB''}(X,\,
j^*j_{*}Y) \cong \Hom_{\AB''}(X,\, Y)\ .
\end{equation}
Letting $ X = Y $ in \eqref{norm}, we get a family of morphisms
$\, \N_X:\, j_! X \to j_* X \,$ in $ \AB $, corresponding to the
identity maps of objects in $ \AB'' $. This defines a
natural transformation $\, \N:\, j_! \to j_* \,$ (see \cite{BBD},
1.4.6.2), satisfying
\begin{equation}
\la{nid} \N\, j^* = \eta \circ \sigma\ .
\end{equation}
The functor $ j_{! *}: \AB'' \to \
AB $ is now defined by\footnote{This functor 
is referred to as the `prolongement interm\'ediare' in \cite{BBD}.}
\begin{equation}
\la{nfun} 
j_{! *}X := \im(\N_X) \subseteq j_* X \ , \quad  X \in {\rm Ob}(\AB'') \ .
\end {equation}
\begin{lemma}
\la{Linj}

\hfill

$(a)$ $ j_{! *} $ gives an equivalence between $ \AB''
$ and the full subcategory of $ \AB $ consisting of objects $ L $
such that $ i^* L = i^! L = 0 $. 
The inverse of $ j_{! *} $ is given by $\, j^* $.

$(b)$ $\ j^* $ and $ j_{! *}$ induce the mutually inverse bijections
$$
\begin{diagram}[small, tight]
\left\{  
\begin{array}{c}
\mbox{\rm isomorphism classes of} \\ 
\mbox{\rm simple objects} \ L \in \AB \ \mbox{\rm with} \ j^* L \not= 0 
\end{array}
\right\} 
& \ \pile{ \rTo^{j^*}\\ \lTo^{j_{! *}} } \ &
\left\{
\begin{array}{c}
\mbox{\rm isomorphism classes of }\\ 
\mbox{\rm nonzero simple objects of}\ \AB'' 
\end{array}
\right\}
\end{diagram}
$$
\end{lemma}
\begin{proof}
$(a)$ It follows from (R2) that $\, j^*j_{! *} \simeq \id_{\AB''} $,
so $ j_{! *} $ is a fully faithful functor.
On the other hand, in view of \eqref{nid}, we have $\, j_{! *} j^* L =
\im(\eta_L \circ \sigma_L)\,$ for any $ L \in \mbox{Ob}(\AB) $. 
If $\, i^* L = i^! L = 0 \,$, the map $ \sigma_L $ is epi and $ \eta_{L}
$ is mono, see  \eqref{seqs}, so $\, \im(\eta_L \circ \sigma_L) \cong L \,$. 
Conversely, (R4) implies $\, i^* j_{! *} = i^! j_{! *} = 0\,$; 
hence, if $\, j_{! *} j^* L \cong L \,$ for some $ L \in
\mbox{Ob}(\AB) $, then $\, i^* L = i^! L = 0 \,$. The (essential)
image of $ j_{! *} $ consists thus of objects 
$ L \in \AB $ with $\, i^* L = i^! L = 0 \,$.

$ (b)\,$ follows immediately from $(a)$ if we observe that the (nonzero)
simple objects $ L \in \AB $ with $\, i^* L = i^! L = 0 \,$
are exactly the ones, for which $ j^*L \not=0$.
\end{proof}

\subsection{Main theorem}
\la{MTh} From now on, we fix $ n \geq 0 $ and let $ \lambda =
(-n,1) $ as in Section~\ref{RV}. The following observation is then
immediate from the definition of $ \Pi_\lambda $.
\begin{lemma}
\la{L} $ A_1 $ is isomorphic to the quotient of $\, \Pi_\lambda \,
$ by the ideal generated by $\, e_\infty \,$.
\end{lemma}
In fact, combined with the canonical projection $\, \Pi_\lambda
\onto \Pi_\lambda /\langle e_\infty \rangle \,$, the algebra map
$\,\C\langle x,y\rangle \to \Pi_\lambda \,$, $\, x \mapsto X \,$,
$\, y \mapsto Y $, is an epimorphism with kernel containing $\
xy-yx-1 \,$. The induced map $\, A_1 := \C\langle x,y\rangle/
\langle xy-yx-1 \rangle \to \Pi_\lambda/\langle e_\infty
\rangle\,$ is then an isomorphism of algebras, since $ A_1 $ is
simple.

Simplifying the notation, we write $\, \Pi = \Pi_{\lambda}\, , \,
A = A_1 $, and let $\, B := e_\infty \Pi \,e_\infty \,$. Clearly,
$ B $ is an associative subalgebra of $ \Pi $ with identity
element $ e_{\infty} $; by analogy with representation theory of
semisimple complex Lie algebras (see Section~\ref{analog} below),
we call it the {\it parabolic subalgebra} of $ \Pi $.

Now, we set $\, \AB := \Mod(\Pi) \,$, $\, \AB' := \Mod(A) \,$ and
$\, \AB'' := \Mod(B) \,$ and define the six functors \eqref{D1}
between these categories as follows. Using the isomorphism of
Lemma~\ref{L}, we first identify $\, A = \Pi/\langle e_\infty
\rangle \,$, and let $\, i_*\,$ be the restriction functor for the
canonical epimorphism $\, i: \Pi \to A \,$. This functor is fully
faithful and has both the right adjoint $\, i^! := \Hom_{\Pi}(A,\,
\mbox{---}\,)\,$ and the left adjoint $\, i^* := \,\mbox{---}
\otimes_{\Pi} A \,$ with adjunction maps satisfying $\, i^* i_*
\simeq \id_{\AB'} \simeq i^!\,i_* \,$. The functor $\, j^*: \AB
\to \AB'' \,$ is then defined by $\, j^* := \mbox{---}
\otimes_{\Pi} \Pi e_\infty \,$. As $ e_\infty \in \Pi $ is an
idempotent, $ j^* $ is exact and has also the right adjoint $\,
j_* = \Hom_{B}(\Pi e_\infty, \,\mbox{---}\,) \,$ and the left
adjoint $\, j_! = \mbox{---} \otimes_{B} e_\infty \Pi \,$
satisfying $\, j^*j_* \simeq \id_{\AB''} \simeq j^* j_! \,$. Now,
if $\, j^* M = 0 \,$ in $ \AB'' $, we have $\, M \otimes_{\Pi} \Pi
e_\infty = M \, e_\infty = 0 \,$, so $\, \langle e_\infty \rangle
\subseteq \Ann_{\Pi}(M)\,$, and therefore $\, M \cong i_* N \,$
for some $ N \in \mbox{Ob}(\AB')$. Summing up, we have the
following
\begin{proposition}
\la{P1} The functors $ (i^*, \,i_*,\, i^{!}) $ and $ (j_!,
\,j^*,\, j_*) $ defined above satisfy the recollement conditions
of Section~\ref{R}.
\end{proposition}
\noindent Our main result can now be stated as follows.
\begin{theorem}
\la{T1} For each $ n \geq 0 $, the composition of functors
\begin{equation}
\la{EF} \Omega:\ \Mod(\Pi) \stackrel{j^*}{\to} \Mod(B)
\stackrel{j_!}{\longrightarrow} \Mod(\Pi) \stackrel{i^!}{\to}
\Mod(A)
\end{equation}
maps {\rm injectively} the set $ \Irr(\Pi, \alpha) $ of
isomorphism classes of simple $ \Pi$-modules of dimension $ \alpha
= (1,n) $ into the set $ \R $ of isomorphism classes of right
ideals of $ A $. If we identify $\, \CC_n = \Rep(\Pi^{\rm \footnotesize opp},
\alpha)/\!/\G(\alpha) \,$ as in Section~\ref{RV}, then the map
induced by $\, \Omega \,$ agrees with the Calogero-Moser map
$\,\omega: \, \CC_n \to \R \,$ constructed in \cite{BW}.
Collecting such maps for all $ n \geq 0 $, we get thus the
bijective correspondence $ \bigsqcup_{n \geq 0} \CC_n
\stackrel{\sim}{\to} \R  \,$.
\end{theorem}

There are several different constructions of the map $ \omega $ in
the literature, the most explicit one being given in \cite{BC}. To
prove the theorem, we will compute $ \Omega $ on simple
$\Pi$-modules and show that the induced map agrees with \cite{BC}.
For this we will need another observation regarding the structure
of $ \Pi_\lambda $.
\begin{lemma}
\la{L22} If $ \lambda = (-n,1) $ with $ n \not= 0 $, the algebra $
\Pi_{\lambda} $ is Morita equivalent to $ e_0 \Pi_{\lambda} e_0 $,
while $ e_0 \Pi_{\lambda} e_0 $ is isomorphic to a quotient of the
free algebra $ R = \C \langle x, y \rangle $.
\end{lemma}
\begin{proof}
As $ e_0 $ is an idempotent in $ \Pi = \Pi_\lambda $, the natural
functor $\,\Mod(\Pi) \to \Mod(e_0 \Pi e_0)\,$, $\, M \mapsto M e_0
\,$, is an equivalence of categories if and only if $\, \Pi e_0
\Pi = \Pi \,$ (cf. \cite{MR}, Prop.~3.5.6). This last identity
holds in $ \Pi\,$, since
$$
1 = e_0 + e_\infty = e_0 - \frac{1}{n}\,vw = e_0 - \frac{1}{n}\,v
\,e_0 \, w \in \Pi e_0 \Pi
$$
in view of the defining relations \eqref{E25}. Now, we have $\, X
,\, Y \in e_0 \Pi e_0 \,$, and it is again easily seen from
\eqref{E25} that the algebra map
\begin{equation}
\la{algmap}
 \pi:\, R \to e_0 \Pi e_0 \ , \quad 1 \mapsto e_0\ ,
\quad x \mapsto X \ , \quad y \mapsto Y \ ,
\end{equation}
is surjective, with $ \Ker(\pi) $ generated by
$\,([x,\,y]-1)([x,\,y]+n-1) \in R \,$.
\end{proof}

\noindent {\sl Proof of Theorem~\ref{T1}.} By definition, the
functors $ i^* $ and $ i^{!} $ assign to a $\Pi$-module $ L $ its
largest quotient and largest submodule over $ A $ respectively. In
particular, both  $ i^* $ and $ i^{!} $ map finite-dimensional
modules to finite-dimensional ones. Since $0$ is the only such
module over $A$, we have $\, i^* L = i^! L = 0 \,$ for any $ L \in
\mbox{\rm Ob}(\AB) $ with $ \dim(L) < \infty $. By
Lemma~\ref{Lexact}$(a)$, $\, \Omega(L) \,$ is then isomorphic to
the kernel of $\, \sigma_L: j_! j^* L = L e_\infty \otimes_B
e_\infty \Pi \to L \,$, which is the obvious multiplication-action
map.

Now, let $L$ have dimension vector $ \alpha = (1,n) $. Set $ M :=
L e_\infty \otimes_B e_\infty \Pi \,$. Since $\, \dim(j^* L) =
\dim(L e_{\infty}) = 1 \,$, $ M $ is a cyclic $ \Pi$-module
generated by $ \xi \otimes_{B} e_\infty $, where $ \xi $ is a(ny)
nonzero vector of $\, j^* L = L e_\infty $. To describe $ M $
explicitly, we will compute the annihilator of $ \xi \otimes_{B}
e_\infty $ in $\, \Pi \,$. First, we observe that, as a subspace
of $ \Pi $, the algebra $ B = e_\infty \Pi \, e_\infty $ is
spanned by the elements $\, v\, a(X,Y) \,w \in \Pi \,$, where $\,
a(x,y) \in R \,$. So we can define on $R$ the linear form
$\,\varepsilon: R \to \C \,$ by
\begin{equation}
\la{func} \xi.(v\, a(X,Y) \,w) =  \varepsilon(a) \,\xi \ .
\end{equation}
(Note that $ \varepsilon $ is independent of the choice of the
basis vector $\, \xi \in j^* L \,$, and it determines the
$B$-module $ j^* L $ uniquely, up to isomorphism.)

It is easy to see now that $\,\Ann_{\Pi}(\xi \otimes_{B}
e_\infty)\,$ is generated by $\, e_0 \,$ and the elements $\,
\{v\, a(X,Y) \,w + \varepsilon(a)\,:\, a \in R\}\,$; thus, with
natural splitting $\,\Pi = e_0 \Pi \oplus e_{\infty} \Pi \,$, we
have the isomorphism of $\Pi$-modules
\begin{equation}
\la{iden} e_\infty \Pi \left/\sum_{a \in R} \left[\,v\, a(X,Y) \,w
- \varepsilon(a) \right]\, e_\infty \Pi \right. \
\stackrel{\sim}{\to} \ M\ , \quad [\,p\,] \mapsto \xi \otimes_B p
\ ,
\end{equation}
where $ \,[\,p\,]\,$ denotes the residue class of $\, p \in
e_\infty \Pi\,$.

On the other hand, the left multiplication by $ v \in \Pi\,$
induces the map
\begin{equation}
\la{iden2} e_0 \Pi \left/\sum_{a \in R} \left[\, a(X,Y) \,w v -
\varepsilon(a)\right] \, e_0 \Pi \right. \to v\,\Pi \left/
\sum_{a\in R}\left[\, v\,a(X,Y) \,w  - \varepsilon(a)\right]\,
v\,\Pi \right.\ ,
\end{equation}
which is also easily seen to be an isomorphism of $\Pi$-modules.

 If $\, n \not= 0 \,$, we have $\, v \,\Pi = e_\infty
\,v\,\Pi \subseteq  e_\infty \Pi = \frac{1}{n} v w \,\Pi \subseteq
v\,\Pi\,$. Whence $\, v\,\Pi = e_\infty \Pi \,$. Combining now
\eqref{iden} and \eqref{iden2}, we get
\begin{equation}
\la{iden1} e_0 \Pi \left/ \sum_{a \in R} \left[\, a(X,Y) \,w v -
\varepsilon(a) \right]\, e_0 \Pi \right. \ \stackrel{\sim}{\to} \
M\ ,
\end{equation}
with $\,[\,e_0] \,$ corresponding to the (cyclic) vector $\, \xi
\otimes_B v \in  M $. Under the Morita equivalence of
Lemma~\ref{L22}, the module $ M $ transforms to $\, M e_0 = L
e_\infty \otimes_B e_\infty \Pi \, e_0 \,$ and the isomorphism
\eqref{iden1} becomes
\begin{equation}
\la{iden3} e_0 \Pi e_0 \left/ \sum_{a \in R} \left[\, a(X,Y) \,w v
- \varepsilon(a) \right]\, e_0 \Pi e_0  \right. \
\stackrel{\sim}{\to} \ M e_0\ .
\end{equation}
Through the algebra extension $\, \pi:\, R \onto e_0 \Pi e_0 \,$,
see \eqref{algmap}, we can regard  $\, M e_0 \,$ as a right $
R$-module and, using \eqref{iden3}, we can then identify it with a
quotient of $ R\,$: more precisely, we have
\begin{equation}
\la{iden4} R/J \ \stackrel{\sim}{\to} \ M e_0\ , \quad [\,1\,]_J
\mapsto  \xi \otimes_B v\ ,
\end{equation}
where $\, J \,$ is the right ideal of $ R $ generated by $\,
\{a\,([x,\,y]-1) - \varepsilon(a)\,:\, a \in R \}\,$.

Now, suppose that $ L $ is a simple module representing a point $
[(\X, \Y, \v, \w)] $ of $ \CC_n $, see Section~\ref{CMS}. Under
the equivalence of Lemma~\ref{L22}, it corresponds to $\, L e_0
\,$, which we can again regard as an $R$-module via
\eqref{algmap}. With identification of Section~\ref{RV}, $\, L e_0
\,$ is then simply $\, \C^n $ with (right) action of $\, x,\, y
\in R \,$ defined by the matrices $\, \X, \Y \in \End(\C^n) \,$.
The associated functional $\, \varepsilon: R \to \C \,$ is given
by $ a(x,y) \mapsto  \w\, a^{\tau}(\X, \Y) \,\v \,$, where $\, a
\mapsto a^\tau $ is the natural anti-involution of the algebra $ R
$ acting identically on $x$ and $y$. Now, with identification
\eqref{iden4}, the canonical map $\, M e_0 \to L e_0 \,$,
corresponding to $ \sigma_L $ under the Morita equivalence,
becomes the homomorphism of $R$-modules
$$
R/J \to \C^n \ , \quad [\,1\,]_{J} \mapsto \v \ .
$$
As shown in \cite{BC}, Theorem~5, the kernel of this last
homomorphism represents an ideal class in $\R$, which is exactly
the image of $ [(\X, \Y, \v, \w)] $ under the Calogero-Moser map
$\, \omega: \CC_n \to \R \,$  of \cite{BW}. Thus we get $\,
[\Omega(L)] = \omega[(\X, \Y, \v, \w)] \,$. Now, if $ \Omega $
maps two simple $ \Pi$-modules $ L $ and $ L' $ to isomorphic
$A$-modules, the corresponding functionals $\, \varepsilon \,$ and
$\, \varepsilon' \,$ on $ R $ must coincide (see \cite{BC},
Theorem~3). Hence, if $\, [\Omega(L)] = [\Omega(L')]\,$, we have
$\, j^* L \cong j^* L' $, and by Lemma~\ref{Linj}$(b)$, we then
conclude $\, L \cong L' $. This finishes the proof of the theorem.

\vspace{1ex}

\begin{remark}
By Lemma~\ref{Lexact}, the functor $ \Omega $ can be 
described on simple modules $ L \in \Mod(\Pi) $ in two equivalent ways: 
either as the kernel of the canonical map $\, \sigma_L: \, j_!
j^* L \to L \,$ (see above) or as $\, (\LL_1 i^*) L \cong
\Tor_{1}^{\Pi}(L,\,A) \,$. This last formula should be compared 
to formula (A.5) in the Appendix of \cite{BW1}.
\end{remark}
\subsection{Analogy with highest weight modules}
\la{analog} We would like to draw reader's attention to an
interesting analogy with representation theory of semisimple
complex Lie algebras. The modules $ M $ and $ L $, which appear in
the proof of Theorem~\ref{T1}, should be compared to the Verma
module $ M(\lambda) $ and its simple quotient $ L(\lambda) $ in
the case when $ \lambda \in \h^* $ is a dominant integral weight
(see \cite{D}, Chapter~7). By analogy with Verma modules, we can
characterize $ M $ as a universal cyclic $\Pi$-module generated by
a ``highest weight'' vector $\, \bxi \in M \,$, satisfying  $\,
\bxi\, e_0 = 0 \,$ and $\, \bxi\,b =  \varepsilon(b)\,\bxi \,$ for
all $\, b  \in B \,$ (cf. \eqref{func} above). The module $ L $ is
then the unique simple quotient of $M$; the subalgebra $ B $ (or
rather, its unital extension $\, \tilde{B} = \C e_0 + B \subset
\Pi \,$) plays a role similar to (the universal enveloping algebra
of) the Borel subalgebra $\, U(\b) \subset U(\g) \,$ in Lie
theory, and $\, \varepsilon \,$ is an analogue of the weight
functional $ \lambda \in \h^*$.

Now, in the Lie case the assertion of Theorem~\ref{T1} amounts to
injectivity of the map $\, [L(\lambda)] \mapsto [K(\lambda)] \,$,
assigning to the isomorphism class of a finite-dimensional simple
$\g$-module $ L(\lambda) $ the isomorphism class of the unique
maximal submodule $ K(\lambda) $ of $ M(\lambda) $ (i.~e., the
kernel of the canonical epimorphism $\, M(\lambda) \to L(\lambda)
\,$). By the classical Cartan-Weyl theory, the isomorphism classes
of finite-dimensional simple $\g$-modules are in bijection with
the set $ {\mathcal P}(\g)_{+} $ of dominant integral weights, and
the map $\, \lambda \mapsto [K(\lambda)] \,$ is indeed injective
on this subset of $ \h^* $. To see this we observe that, for any $
\lambda \in \h^* $, the module $ K(\lambda) $ admits a (unique)
central character $ \chi_{\lambda} $ equal to the central
character of the Verma module $ M(\lambda) $, of which it is a
submodule. Thus, if $\, K(\lambda) \cong K(\lambda') \,$ for some
$\, \lambda, \, \lambda' \in \h^* $, we have $\, \chi_{\lambda} =
\chi_{\lambda'} \,$, and therefore $\, \lambda + \varrho \,$ and $
\,\lambda' + \varrho \,$ lie on the same orbit (in $ \h^* $) of
the Weyl group $ W $ associated to $\g\,$. Now, if $\lambda $ and
$ \lambda' $ are both dominant, then $\, \lambda + \varrho \,$ and
$\, \lambda' + \varrho \,$ are regular dominant: that is, $\,
\lambda + \varrho \,$ and $\, \lambda' + \varrho \,$ belong to the
same (positive) Weyl chamber of $ W $ and hence coincide. In this
way, for $\, \lambda, \, \lambda' \in {\mathcal P}(\g)_{+} $, we
have $\, K(\lambda) \cong K(\lambda') \,\Rightarrow\, \lambda =
\lambda' $.

%

\section{Simple Modules of Cherednik Algebras}
\la{Cher} In this section we will deal with rational Cherednik
algebras $\, H_{0,c}(S_n) \,$ associated to a symmetric group $
S_n $, with deformation parameters $ t = 0 $ and $ c \not= 0 $
(see \cite{EG}). By a result of Etingof and Ginzburg, the
isomorphism classes of simple modules over $ H_{0,c}(S_n) $ are
parametrized by points of the Calogero-Moser space $ \CC_n $. Our
goal is to relate these modules to ideals of $ A_1$ functorially.
We will do this in two steps: first, we construct a functor $
\Xi:\, \Mod(H) \to \Mod(\Pi) $, inducing a natural bijection
between simple $ H$-modules and simple $\Pi$-modules of dimension
$ \alpha = (1, n)$; then we will apply Theorem~\ref{T1}, combining
$ \Xi $ with the functor $ \Omega $.

\subsection{Cherednik algebras and their spherical subalgebras}
Recall that, for a fixed integer $ n \geq 1 $ and a parameter $ c
\in \C \,$, $\, H_{0,c}(S_n) \,$ is generated by two polynomial
subalgebras $ \C[x_1, x_2, \ldots, x_n]  \,$, $\, \C[y_1, y_2,
\ldots, y_n] $ and the group algebra $ \C S_n  $ subject to the
following ``deformed crossed product'' relations (cf. \cite{EG},
(4.1)):
\begin{equation}
\la{rel1} s_{ij} \,x_i = x_{j} \, s_{ij} \ , \quad s_{ij} \, y_{i}
= y_{j} \, s_{ij} \ ,
\end{equation}
\begin{equation}
\la{rel2} [x_i,\, y_j] =  c\, s_{ij} \ \  (i \not= j) \ , \quad
[x_k,\,y_k] = - c\, \sum_{i\not= k}\, s_{ik} \ .
\end{equation}
(Here $ s_{ij}  $ denote the elementary transpositions $ \, i
\leftrightarrow j \,$ generating $ S_n $.) It is easy to see that
the algebras $\, H_{0,c}(S_n)$  with $ c \not=0 $ are all
isomorphic to each other, and it will be convenient for us to fix
$ c = 1 $. Thus we write $ H = H_{0,1}(S_n) $ and let $\, U := e H
e \,$, where $\, e := \frac{1}{n!} \sum_{\sigma \in S_n}\!\sigma
\,$ is the symmetrizing idempotent in $\, \C S_n \subset H \,$. By
definition, $\,U\,$ is an associative subalgebra of $ H $ with
identity element $ e \,$; following \cite{EG}, we call it the {\it
spherical subalgebra}. Two properties of $ U $, which play a
crucial role in representation theory of Cherednik algebras, are
given in the following theorem.
\begin{theorem}[see \cite{EG}, Theorems~1.23 and~1.24]
\la{Sph}

\hfill

$(a)$\ $ U $ is Morita equivalent to $ H $, the equivalence $
\Mod(H) \to \Mod(U) $ being the natural functor $\, e\,:\, M
\mapsto Me \,$;

$(b)$\ $ U $ is a commutative algebra isomorphic to $ \O(\CC_n)
\,$, the coordinate ring of the Calogero-Moser variety $ \CC_n $.
\end{theorem}

Note that it is a consequence of Theorem~\ref{Sph} that the
isomorphism classes of simple $H$-modules are in bijection with
(closed) points of the variety $ \CC_n $.

\subsection{The functor $ \Xi $}
\la{Xi} There seem to be no natural maps relating the algebras $
\Pi $ and $ H \,$; however, at the level of their subalgebras $ B
$ and $ U $ such a map does exist.
\begin{lemma}
\la{Lee2} The assignment $\, v \,a(X,Y)\, w\, \mapsto\,  -
\sum_{i=1}^{n} e\, a(x_i, y_i)\, e \,$ defines a homomorphism of
(unital) algebras $\, \theta:\, B \to U \,$.
\end{lemma}
\begin{proof} A straightforward calculation using the defining relations
of $ H $ and $ \Pi \,$.
\end{proof}

Let $ \theta_*: \Mod(U) \to \Mod(B) $ be the restriction functor
corresponding to the map $ \theta $ of Lemma~\ref{Lee2}. Recall
the functor $ j_{! *}: \Mod(B) \to \Mod(\Pi) $ defined in
Section~\ref{R}.
\begin{theorem}
\la{T2} The composition of functors
$$
\Xi\,:\ \Mod(H) \stackrel{e}{\simeq} \Mod(U)
\stackrel{\theta_*}{\longrightarrow} \Mod(B) \stackrel{j_{!
*}}{\longrightarrow} \Mod(\Pi)
$$
maps the set $ \Irr(H) $ of isomorphism classes of simple $
H$-modules {\rm bijectively} into the set $ \Irr(\Pi, \alpha)$ of
isomorphism classes of simple $\Pi$-modules of dimension $ \alpha
= (1,n)$. If we identify $\, \CC_n = \Rep(\Pi^{\footnotesize \rm opp},
\alpha)/\!/\G(\alpha) \,$ as in Section~\ref{RV}, then the map
$\, \Irr(H) \to \Irr(\Pi, \alpha) \,$ induced by $ \Xi $ agrees
with the Etingof-Ginzburg map in \cite{EG}.
\end{theorem}
\begin{proof}
We will compute $ \Xi $ on simple $ H$-modules and show that the
induced map agrees with \cite{EG}. We begin by recalling the
Etingof-Ginzburg construction (see {\it loc.\! cit.}, Section~11).

Let $  S_{n-1} \subset S_n  $ be the subgroup of permutations of
$\, \{1,\,2,\,\ldots,\, n\}\,$ acting trivially on $\, 1 \,$. Let
$\, \e := \frac{1}{(n-1)!} \sum_{\sigma \in S_{n-1}}\!\sigma \,$
be the corresponding idempotent in $ \C S_n \subset H $. Given a
simple $H$-module $ V $, set $\, \V := V \e \subseteq V \,$. It is
known (see \cite{EG}, Theorem~1.7) that $\, V \cong \C S_n \,$ as
an $S_n$-module, so $\,\dim(V) = n!\,$ and $\,\dim(\V) = n\,$. On
the other hand, in view of the defining relations \eqref{rel1},
the elements $ x_1 $ and $ y_1 $ commute with $ \e $ in $H$, and
the subspace $ \V $ is therefore stable under their action on $
V\,$. If we define now  $ \X \in \End(\V) $ and $ \Y \in \End(\V)
$ by restricting the action of $ x_1 $ and $ y_1 $ to $ \V $, we
get $\, \rk([\X,\,\Y] + \id) = 1 \,$. Indeed, using \eqref{rel2},
it is easy to see that
\begin{equation}
\la{rank} \e\,([x_1,\, y_1] - 1) = -n \, e\quad \mbox{in}\ H\ .
\end{equation}
Whence $\, \im([\X,\,\Y] + \id) \subseteq \V e = V e \,$. On the
other hand, by part $ (a) $ of Theorem~\ref{Sph}, $\, V e $ is a
simple $U$-module, and by part $(b)$, it is then $1$-dimensional.
Thus, we have $\, [\X,\,\Y] + \id = - \v\,\w\,$ for some $\,\v \in
\V $ and $\, \w \in \V^*$. It follows that  $\, (\V; \,\X,\, \Y,\,
\v,\,\w) \,$ represents a point of the variety $ \CC_n $. Now, let
$ L $ be a simple $\Pi$-module corresponding to this point under
the identification of Section~\ref{RV}. Restricting $ L $ to the
parabolic subalgebra $ B = e_{\infty} \Pi\,e_{\infty} $, we get
the $1$-dimensional $ B$-module $ j^*L = L e_{\infty} $. As
mentioned in the proof of Theorem~\ref{T1}, the module $ j^{*}L $
is determined, up to isomorphism, by its ``weight'' functional
$\,\varepsilon:\,R \to \C \,$, see \eqref{func}. The latter can be
written explicitly in terms of the Calogero-Moser matrices: $\,
\varepsilon(a) = \w\,a^{\tau}(\X,\,\Y)\,\v \,$, where $\, \tau\,$
is the canonical anti-involution on $R$ satisfying $ x^{\tau} = x
$ and $ y^\tau = y $. Since
$$
\w\,a^{\tau}(\X,\,\Y)\,\v = \tr_{\V}[\,a^{\tau}(\X,\,\Y)\,\v\w\,]
= -\, \tr_{\V}[\,a^{\tau}(\X,\,\Y)\,([\X,\,\Y] + \id)\,]\ ,
$$
using \eqref{rank}, we get
\begin{equation}
\la{ep}
 \varepsilon(a) = - n \,\tr_{\V}[\,\cdot \,e\,a(x_1,\,y_1)] = -
 n\,\tr_{V}[\,\cdot \, e\,a(x_1,\,y_1)]\ ,
\end{equation}
where ``$\,\cdot\,$'' denotes the action of $ H $ on $ V $.

Now, let us regard $\, Ve \,$ as a $ B$-module via the algebra map
$ \theta $. As $V$ is simple, $\,Ve$ is $1$-dimensional, its
``weight'' functional $\,\tilde{\varepsilon}:\, R \to \C \,$ being
\begin{equation}
\la{ep1} \tilde{\varepsilon}(a)  =  - \sum_{i=1}^{n} \tr_{V\! e}
[\,\cdot\, e\, a(x_i, y_i)\, e]\ .
\end{equation}
By symmetry, we have  $\, e\, a(x_i, y_i)\, e = e\, a(x_1, y_1)\,
e \,$ in $ H $ for any $ i = 1,2,\ldots, n\,$, so the sum in the
right-hand side of \eqref{ep1} equals $\, - n\,\tr_{V\!e}
[\,\cdot\, e\, a(x_1, y_1)\, e] \,$. Comparing this to \eqref{ep},
we see that $\, \tilde{\varepsilon} = \varepsilon\,$. Hence $\,
\theta_*(Ve) \cong j^* L \,$ as $B$-modules. Now, by
Lemma~\ref{Linj}$(a)$, we have $\, \Xi(V) = j_{! *} \theta_*(Ve)
\cong j_{! *}j^*L \cong L\,$, which shows that the
Etingof-Ginzburg map is indeed induced by the functor $ \Xi $.
\end{proof}

Combining Theorems~\ref{T1} and \ref{T2} together, we get our
second main result.
\begin{theorem}
\la{T3} The composition of functors
\begin{equation}
\la{HA} \Mod(H) \stackrel{e}{\simeq} \Mod(U)
\stackrel{\theta_*}{\longrightarrow} \Mod(B)
\stackrel{j_!}{\longrightarrow} \Mod(\Pi) \stackrel{i^!}{\to}
\Mod(A)
\end{equation}
maps the set $ \Irr(H) $ of isomorphism classes of simple $
H$-modules {\rm injectively} into the set $ \R $ of isomorphism
classes of right ideals of $ A $. Collecting such maps for all $ n
\geq 0 $, we get a bijective correspondence $ \bigsqcup_{n \geq 0}
\Irr [H(S_n)] \stackrel{\sim}{\to} \R  \,$.
\end{theorem}

\section{Quiver Generalization}
\la{QG} The above results generalize to an arbitrary affine Dynkin
quiver associated (via McKay's construction) to a finite subgroup
$\, \Gamma \,$ of $\, \SL(2, \C)\,$. The Weyl algebra $ A_1 $ is
replaced in this situation by a ``quantized coordinate ring'' $\,
O_{\tau}(\Gamma) \,$ of the Kleinian singularity $\,
\C^2\!/\!/\Gamma \,$ and the rational Cherednik algebra $
H_{0,c}(S_n) $ by a symplectic reflection algebra $
H_{0,k,c}(\bGamma_n) $ associated to the $n$-th
wreath product $\, \boldsymbol{\Gamma}_n = S_n \ltimes
\Gamma^n\,$.

In what follows we outline this generalization in the special case
of cyclic groups. The construction of $\A$-envelopes of ideals of
$\, O_{\tau}(\Gamma) \,$ for such $\Gamma$'s, using the
approach of \cite{BC}, has been carried out recently in \cite{E}.
The results of \cite{E} can be reinterpreted in terms of
representation theory of deformed preprojective algebras of Dynkin
quivers of type $ \tilde{A}_m $, and (the analogues of)
Theorems~\ref{T1}, \ref{T2} and \ref{T3} hold true in this case.
The proofs repeat almost verbatim the above arguments for the Weyl
algebra, so we omit them for the sake of brevity. The case of
general (noncyclic) Dynkin quivers seems to be more instructive
and will be treated in detail elsewhere.

\subsection{Deformations of Kleinian singularities}
Let $(L, \omega)$ be a two-dimensional complex symplectic vector
space with symplectic form $\omega$, and let $\Gamma $ be a finite
subgroup of $ \SSp(L,\omega)$. The natural action of $\Gamma$ on $
L^* = \Hom(L, \C) $ extends diagonally to the tensor algebra
$TL^*$, and we write $ R := TL^* \# \Gamma $ for the crossed
product of $ TL^* $ with $\Gamma $. Since $\, \omega \in L^*
\otimes L^* \subset TL^{*} \subset R $, we can form the (two-sided)
ideals $\, \langle \omega - \tau \rangle \,$ in $ R $, one for
each $ \tau\in \C \Gamma $, and define $\, \S_\tau(\Gamma) :=
R/\langle \omega - \tau \rangle \,$. Furthermore, taking the
idempotent $\, \ee := \frac{1}{|\Gamma|}\, \sum_{g \in \Gamma} g \in
\C\Gamma \subset \S_\tau(\Gamma) $, we set $\, O_\tau(\Gamma) :=
\ee\,\S_\tau(\Gamma)\,\ee \,$. By definition, $ O_\tau(\Gamma)$ is a
unital associative algebra, with $ \ee $ being the identity element.
For $ \tau = 0 $, it is a commutative ring isomorphic to $
\C[L]^{\Gamma} $. The family of algebras $
\{O_\tau(\Gamma)\}_{\tau \in \C \Gamma} $ can thus be viewed as a
(noncommutative) deformation of $ \C[L]^\Gamma $, the coordinate
ring of the classical Kleinian singularity $ L/\!/\Gamma $, see \cite{CBH}.

As mentioned above, we will be dealing here only with cyclic
$ \Gamma$'s. Thus we fix $ m \geq 1 $ and assume $\, \Gamma \cong \Z/\Z
m $. In this case we can choose a basis $ \{x, y\} $ in $ L^* $,
so that $\, \omega = x \otimes y - y \otimes x\,$ and $ L^* $
decomposes as $ \chi \oplus \chi^{-1} $, with $
\Gamma $ acting on $ x $ by $ \chi $ and on $y$ by
$\chi^{-1}$, where $ \chi $ is a primitive character
of $ \Gamma$. The algebra $ \S_\tau(\Gamma) $ can then be identified
with a quotient of $ \C\langle x, y \rangle \# \Gamma $, with $\, x, y
\, $ and $ g\in\Gamma $ satisfying the relations
\begin{equation}
\la{relS} g \cdot x = \chi (g)\, x \cdot g \ , \quad g
\cdot y = \chi^{-1}(g)\, y\cdot g \ , \quad x \cdot y - y
\cdot x = \tau \ .
\end{equation}

Following \cite{CBH}, we can give another construction of
$ \S_\tau(\Gamma) $. Recall that with each finite subgroup
$ \Gamma \subset \mathtt{SL}(2, \mathbb C) $ one can associate
a quiver $\, Q(\Gamma) \,$, whose underlying graph is an extended
Dynkin diagram (see \cite{M}). In case when $\, \Gamma \cong \Z/\Z m \,$,
the quiver $\, Q(\Gamma) \,$ has type $ \tilde{A}_m $:  it consists of $ m $
vertices, indexed by $\, \{\,0,\, 1,\, ...\,,\,m-1 \,\} \,$, and $ m $
arrows $ \{\,X_0,\, X_1,\, ...\,,\,X_{m-1}\}\,$, forming a cycle,
see \eqref{D44}.
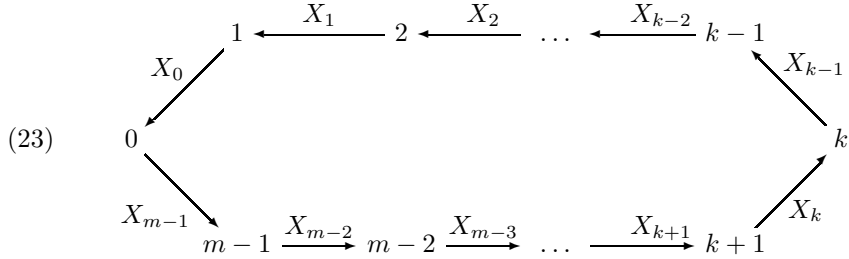
\begin{equation}
\la{D44}
\begin{diagram}[small, tight]
   &               &1  &\lTo^{X_1}&2&\lTo^{X_2}&\ \hdots \ &  &\lTo^{X_{k-2}}& k-1 &                &  \\
   &\ldTo^{X_0}    &     &          &   &          &           &  &              &     & \luTo^{X_{k-1}}&   \\
0&               &     &          &   &          &           &  &              &     &                & k\\
   &\rdTo_{X_{m-1}}&     &          &   &          &           &  &              &     & \ruTo_{X_k}    &   \\
   &               & m-1 &\rTo^{X_{m-2}}& m-2 &\rTo^{X_{m-3}}&\  \hdots  \ & &\rTo^{X_{k+1}}& k+1 &     &   \\
\end{diagram}
\end{equation}
\vspace{0.1cm}

Given $ \tau = (\tau_0, \,\tau_1,\, \ldots\,,\, \tau_{m-1}) \in
\C^m $, we consider the deformed preprojective algebra $\,
\Pi_{\tau}(Q)\,$ of weight $ \tau $. Explicitly, this algebra
is generated by $ 2m $ arrows $\, X_k,\, Y_k\,$ and $ m $
vertices $ \, e_0,\, e_1,\, ...,\, e_{m-1} \,$, which apart from
the standard path algebra relations, satisfy
\begin{equation}
\la{relP1}
 X_{0} Y_{0}  - Y_{m-1} X_{m-1} = \tau_0 \,e_{0}
\end{equation}
and
\begin{equation}
\la{relP2} X_{k} Y_{k} - Y_{k-1} X_{k-1} =\tau_{k}\, e_{k} \ ,
\end{equation}
where $\, k = 1,\,2,\, \ldots\, , \, m-1\,$.
Now, if we identify $\, \C \Gamma $ with $\C^m \,$ by choosing a basis
$\{ \ee_i \} $ in $ \C \Gamma$ with
\begin{equation}
\la{idd}
\ee_i := \frac{1}{|\Gamma|}\, \sum_{g \in \Gamma} \chi^{i}(g)\,g
\ , \quad  i = 0,1, \ldots, m-1 \ ,
\end{equation}
then the mapping
\begin{equation}
\la{map}
\ee_i \mapsto e_i \ , \quad
x \mapsto \sum_{i=0}^{m-1} X_i\ ,\quad
y \mapsto \sum_{i=0}^{m-1} Y_i\ ,
\end{equation}
extends to an algebra isomorphism
\begin{equation}
\la{PS}
\S_{\tau}(\Gamma) \stackrel{\sim}{\to} \Pi_{\tau}(Q) \ .
\end{equation}
Indeed, it is easy to check, using the relations \eqref{relS} and
\eqref{relP1}, \eqref{relP2}, that \eqref{map} yields a well-defined
algebra map, and its inverse is given by $\, X_i \mapsto \ee_i x \,$,
$\, Y_i \mapsto \ee_i y\,$.

\vspace{1ex}
\begin{remark}
For an arbitrary $ \Gamma \subset \SL(2, \C) $, the relationship
between the algebras $ S_\tau(\Gamma) $ and $ \Pi_\tau(Q) $ is weaker:
one has only a Morita equivalence rather than an isomorphism
(see \cite{CBH}, Theorem~0.1).
\end{remark}

\vspace{1ex}

Next, we recall that there is  a root system $ \DD(Q) \subset \Z^I $
attached to any finite quiver $ Q = (I,\,Q) $, see \cite{Kac}.
If $ Q $ is an (extended)
Dynkin quiver of type A, D or E, the corresponding $ \DD(Q) $ is
the usual (affine) root system of the same type; in particular,
for the cyclic quiver \eqref{D44}, $\,\DD(Q) \,$ is the affine root
system of type $ \tilde{A}_m $. In this last case, we say that
$\, \tau \in \C^m $ is {\it regular},
if $\, \tau \cdot \alpha \not= 0\,$ for all $\, \alpha \in \DD(Q) \,$.

The following proposition is one of the main results of \cite{CBH}
(see also \cite{H, S}).
\begin{proposition}[ \cite{CBH}, Theorem~0.4]
\la{PCH}
If $ \tau $ is regular, then $\,\S_{\tau}(\Gamma)\,$ and $\, O_{\tau}(\Gamma) \,$
are simple rings, Morita equivalent to each other.
\end{proposition}

\subsection{Nakajima varieties}
In case of nontrivial $\Gamma$'s, the Calogero-Moser spaces $ \CC_n $ are
replaced by Nakajima quiver varieties \cite{N, N1}. We define these
varieties as representation spaces of ``framed'' Dynkin quivers, following
\cite{CB3}.

First, as in Section~\ref{DPA}, we introduce a quiver $ Q_\infty(\Gamma) $
by adding to $ Q(\Gamma)$ an extra vertex $ \, \infty\, $ and an extra arrow
$\, v:\, 0 \to \infty \,$. Given $\, \lambda = (\lambda_\infty,\,\lambda_0,
\,\ldots ,\, \lambda_{m-1}) \in \C^{m+1} $, we write
$\, \Pi_{\lambda} := \Pi_{\lambda}(Q_\infty)\,$ for the deformed
preprojective algebra of weight $ \lambda $. Explicitly, this algebra
is generated by $
2(m+1) $ arrows $\, X_k,\, Y_k, \,v, \, w\,$ and $
m+1 $ vertices $ \, e_{\infty}, \,e_0,\, e_1,\, ...,\, e_{m-1} \,
$, which apart from the standard path algebra relations, satisfy
\begin{equation}
\la{relP11}
v w = \lambda_{\infty}\, e_{\infty} \ , \quad X_{0} Y_{0}  -
Y_{m-1} X_{m-1} - wv =\lambda_{0} \,e_{0}
\end{equation}
and
\begin{equation}
\la{relP12}
X_{k} Y_{k} - Y_{k-1} X_{k-1} =\lambda_{k}\, e_{k} \ ,
\end{equation}
where $\, k = 1,\,2,\, \ldots\, , \, m-1\,$.

The relation between the algebras $ \Pi_\lambda(Q_\infty) $ and $ \Pi_\tau(Q) $
is given by the following lemma, which is a generalization of Lemma~\ref{L}.
\begin{lemma}
\la{Le1}
If $ \tau = (\lambda_0,\,\lambda_1,
\,\ldots ,\, \lambda_{m-1}) \in \C^{m} $, the algebra $ \Pi_{\tau}(Q) $
is  isomorphic to the quotient of $ \Pi_{\lambda}(Q_\infty) $
by the idempotent ideal generated by $ e_{\infty} $.
\end{lemma}
\begin{proof}
One can show this easily by using the explicit presentations
\eqref{relP1}-\eqref{relP2} and \eqref{relP11}-\eqref{relP12}. There is,
however, a more conceptual way to construct an isomorphism
\begin{equation}
\la{PS1}
\Pi_\tau(Q) \cong \Pi_\lambda(Q_\infty)/\langle e_\infty \rangle \ .
\end{equation}
Consider the natural projection $\, \pi : \C Q_\infty \onto \C Q $,
which maps all paths of $ Q_\infty $ ending at $ \infty $ to zero.
Clearly, $\, \Ker\,\pi \,$ is
generated by $ e_\infty $, and hence it is an idempotent ideal in
$ \C Q_\infty $. It follows that $\,\Tor^{\C Q_\infty}_{1}( \C Q, \C Q)
\cong \Ker\,\pi/(\Ker\,\pi)^2 = 0 \,$, so $ \theta $ is a pseudoflat
ring epimorphism. By \cite{CB2}, Theorem~0.7, it extends then canonically
to an algebra map $ \tpi:\, \Pi_{\lambda}(Q_\infty) \to \Pi_\tau(Q) $,
making the commutative diagram
\begin{equation}
\la{D45}
\begin{diagram}[small, tight]
\C Q_\infty & \rTo^{\pi} & \C Q \\
\dTo& \mbox{\quad} & \dTo \\
\Pi_\lambda(Q_\infty) & \rTo^{\tpi} & \Pi_{\tau}(Q) \\
\end{diagram}
\end{equation}

\noindent
This diagram is a pushout in the category of rings,
so the surjectivity of $ \pi $ implies the surjectivity
of $ \tpi $, and $ \Ker\,\tpi = \langle e_{\infty} \rangle $.
\end{proof}

The next proposition is a special case of \cite{CB2}, Theorem~1.2;
for reader's convenience, we will sketch a direct proof of it, 
based on Lemma~\ref{Le1}.
\begin{proposition}
\la{exist}
Let $\, \tau \in \C^m $ be a {\rm regular} vector for the Dynkin
quiver \eqref{D44}, and let $\, \alpha = (1,\,\bn)\in \Z^{m+1} $.
Then there exists a $ \Pi_\lambda$-module of dimension
vector $ \alpha $ if and only if $\, \lambda =
(-\tau \cdot \bn, \, \tau) \,$ and $\, \alpha\,$ is a
positive root of $ Q_\infty $. Every $ \Pi_\lambda$-module
of dimension vector $ \alpha $ is simple.
\end{proposition}
\begin{proof}
We start with the last assertion. Assume that
there is a $\Pi_{\lambda}$-module $ V $ of dimension $ \alpha $.
Let $\, V' \subseteq V \,$ be a submodule of $ V $.
Consider the intersection $\, V' \cap V_\infty \,$, where
$ V_\infty := V e_\infty $. Since
$ \dim(V_\infty)  = 1 $, there are two possibilities: either
$\, V' \cap V_\infty = 0 \,$ or
$\, V' \cap V_\infty = V_\infty \,$. In the first case we have
$\, V' e_\infty = 0 \,$, so $ V' $ can be viewed as a module over
the quotient algebra $\, \Pi_{\lambda}/ \langle e_\infty \rangle \,$.
By Lemma~\ref{Le1}, this last algebra is isomorphic to $ \Pi_\tau(Q) $,
which, in view of regularity of $ \tau\,$, is a simple ring,
having no non-trivial finite-dimensional modules (see Proposition~\ref{PCH}).
Thus $\, V' \cap V_0 = 0 \,$ implies $\, V' = 0 \,$. In the second case,
if $\, V' \cap V_\infty = V_\infty \,$, we have $\, V e_\infty = V_\infty
\subseteq V' $, so $ e_\infty $ acts trivially on the quotient $ V/V' $.
Applying the above argument to $ V/V' $ instead of $ V' $ yields $ V' = V $.
Thus, in any case $ V' $ is $\, 0 \,$ or $\, V \,$, which means that
$ V $ is a simple module.

Now, the existence of a $ \Pi_\lambda$-module $V$ implies that
$\, \lambda \cdot \alpha = 0 \,$, or equivalently $\, \lambda =
(-\tau \cdot \bn, \, \tau) \,$. The simplicity of $ V $ forces it
to be indecomposable as a $ \C Q_\infty$-module.
A well-known theorem of Kac \cite{Kac} then ensures that $ \alpha $
is a positive root of $ Q_\infty $. Conversely, if $ \alpha $ is
a positive root of $ Q_\infty $, then (again by Kac's Theorem) there
is an indecomposable $ \C Q_\infty$-module. If, in addition,
$\, \lambda \cdot \alpha = 0 \,$, this module extends
to a $ \Pi_\lambda$-module by \cite{CBH}, Theorem~4.3.
\end{proof}

Now, for $\, \tau \in \C^m $ and $\, \alpha = (1, \bn) \in \Z^{m+1} $,
we set
$$
\tilde{\CC}_{\bn, \tau}(\Gamma) :=
\Rep(\Pi_{\lambda}(Q_\infty)^{\rm \footnotesize opp},\, \alpha)\ ,
\quad \lambda = (-\tau \cdot \bn, \, \tau) \ .
$$
By Proposition~\ref{exist}, if $ \tau $ is regular and
$ \alpha $ is a positive root of $ Q_\infty $, the variety
$\,\tilde{\CC}_{\bn, \tau}(\Gamma)\,$ is non-empty, and
the natural action of the group
$$
\G(\alpha) := \left(\C^* \times \prod_{i=0}^{m-1} \GL(n_i,\, \C)\right)
\left/\C^* \cong  \prod_{i=0}^{m-1} \GL(n_i,\, \C)
\right.
$$
on $\,\tilde{\CC}_{\bn, \tau}(\Gamma)\,$ is free. In that case, we define
\begin{equation}
\la{Nak}
\CC_{\bn, \tau}(\Gamma) := \tilde{\CC}_{\bn, \tau}(\Gamma)/\!/\G(\alpha)\ .
\end{equation}
Thus, by definition, $\, \CC_{\bn, \tau}(\Gamma) \,$
are affine varieties, whose (closed) points are in bijection with
isomorphism classes of simple $ \Pi_\lambda$-modules of dimension
vector $ \alpha $. If $ \Gamma $ is trivial and $ \tau = 1 $,
these varieties coincide with the usual Calogero-Moser spaces
$ \CC_n $, see Section~\ref{RV}. In general, they can be
identified with Nakajima's {\it quiver varieties}
$\,{\mathfrak M}^{\tau}_{\Gamma}(V,W)\,$ with
$$
V = W_0^{\oplus n_0} \oplus  W_1^{\oplus n_1} \oplus ...
\oplus W_{m-1}^{\oplus n_{m-1}}\quad \mbox{and}\quad W = W_0\ ,
$$
where $ \, \{W_0,\, W_1,\, ...\,, \,W_{m-1}\} \,$ is a complete set 
of irreducible representations of $\, \Gamma $ and $\, W_0 $ is its
trivial representation (see \cite{CB3}).

\subsection{The generalized Calogero-Moser correspondence}
We are now in position to state our generalization of Theorem~\ref{T1}.
First, we fix $\,\tau \in \C^m$ and identify $\, \S_\tau(\Gamma) =
\Pi_{\lambda}(Q_\infty)/\langle e_\infty \rangle \,$ by combining the
isomorphisms \eqref{PS} and \eqref{PS1}. We denote by
$\, i: \Pi_\lambda \to \S_\tau \,$ the canonical projection
and associate to it the triple of adjoint functors $\,(i^*,\, i_*,\, i^!)\,$,
relating $ \Mod(\Pi_\lambda) $ and $ \Mod(\S_\tau) $, as in Section~\ref{MTh}.
Next, we set $\, B_\lambda := e_\infty \Pi_\lambda e_\infty \,$ and introduce
the functors $ (j_!, \, j^*,\, j_*) $ between the categories
$ \Mod(\Pi_\lambda) $ and $ \Mod(B_\lambda) $. As explained in Section~\ref{MTh},
 $\,(i^*,\, i_*,\, i^!)\,$ and $ (j_!, \, j^*,\, j_*) $ satisfy
the recollement axioms (R1)--(R5). Finally, we introduce the functor
$\, \ee: \Mod(\S_\tau) \to \Mod(O_\tau) \,$, $\, M \mapsto M \ee \,$,
which, in the case of regular $ \tau $, is an equivalence of
categories (see Proposition~\ref{PCH}).
\begin{theorem}
\la{TT1}
Let $\, \tau \in \C^m $ be a regular vector for the Dynkin quiver \eqref{D44},
let $\, \alpha = (1,\,\bn)\in \Z^{m+1} $ be a positive root for $ Q_\infty $,
and let $\, \lambda := (-\tau\cdot \bn,\, \tau)\,$. Then the composition
of functors
$$ \Omega_\tau :\,
\Mod(\Pi_\lambda) \rTo{j^{\ast}} \Mod(B_\lambda) \rTo{j_{!}}
\Mod(\Pi_\lambda) \rTo{i^{!}} \Mod(\S_\tau) \rTo{\ee} \Mod(O_{\tau})
$$
maps injectively the set $ \Irr(\Pi_\lambda,\, \alpha) $ of
isomorphism classes of simple $ \Pi_\lambda$-modules of dimension
vector $ \alpha $ into the set  $ \RR_\tau $ of isomorphism classes of (right)
ideals of $ O_\tau $. If we identify $ \Irr(\Pi_\lambda,\, \alpha)  =
\CC_{\bn, \tau}(\Gamma) $ as above, then the induced map
agrees with the bijections constructed in \cite{BGK1} and \cite{E}.
\end{theorem}
\begin{remark}
If we choose $ \bn = (n,\,n,\,\ldots\,,\,n) \,$ with $\, n \in \NN \,$, then
the functors $ \Omega_\tau $ induce a bijection between
$\, \bigsqcup_{n \geq 0} \CC_{\bn, \tau}(\Gamma) \,$ and the set $
\RR^{o}_{\tau} \subset \RR_{\tau} $ of isomorphism classes 
of {\it stably free} ideals of $ O_\tau(\Gamma) $,
in agreement with the original conjecture of Crawley-Boevey and
Holland (see \cite{CB1}, p.~45).
\end{remark}

\subsection{Relation to symplectic reflection algebras}
In this section we will state the generalizations of 
Theorems~\ref{T2} and~\ref{T3} to the case of cyclic quivers. 
As mentioned above, the rational Cherednik algebras $ H_{0,c}(S_n) $ 
are replaced in this case by more general symplectic reflection
algebras $ H_{0,k,c}(\bGamma_n) $. We recall the definition of these
algebras, following essentially \cite{EG}; we should warn the reader, 
however, that our notation slightly differs from the notation of \cite{EG}.

Let $ \Gamma $ be a cyclic subgroup of $ \SL(2, \,\C) $ of order $ m > 1 $
with a fixed generator $ \alpha \in \Gamma $ and a primitive character $ \chi $.
For an integer $ n \geq 1 $,  let $ S_{n}$ be the permutation group of
$ \{1,..., n\} $ with elementary transpositions $\, s_{ij}:
i \leftrightarrow j$, and let $ \bGamma_n $ be the wreath product
of $ \Gamma $ and $ S_n $, i.~e. the semidirect product
$\, S_{n} \ltimes \Gamma^n \, $ relative to the
obvious action of $ S_n $ on $ \Gamma^n = \Gamma \times \ldots \times \Gamma $.

Let $ L = \C^2 $ be the natural representation of $ \Gamma $ with a
basis $ \{x,y\} $, such that $\alpha(x) = \epsilon x $ and
$\alpha(y) = \epsilon^{-1} y$, where $ \epsilon := \chi(\alpha) $.
The group $\bGamma_n $ acts
naturally on $ L^n$. Given $\, i \in \{1, ..., n \}\,$ and $ \gamma \in \Gamma
$, resp. $ u \in L $, we write $ \gamma_{i} \in \bGamma_n $ for
$ \gamma $ placed in $i$-th factor of $\Gamma^n $, resp. $ u_{i} $
for $ u $ placed in the $ i $-th factor of $ L^n $.

Now, fix parameters $\, k, c_1, \ldots, c_{m-1} \in \C \,$, and let
$ c := \sum_{l=1}^{m-1} c_{l} \alpha^{l} \in \mathbb C \Gamma$.
The {\it symplectic reflection algebra} $\, H_{k,c} :=H_{0,k,c}(\bGamma_{n}) \,$
is then defined as the quotient of $\, T(L^{n}) \# \bGamma_n \,$ by the
following relations
\begin{eqnarray}
\la{sra}
&& [x_{i}, x_{j}] = 0, \quad [y_{i}, y_{j}] = 0 \ , \quad 1 \le i,j \le n\ ,
\nonumber \\*[1ex]
&& [y_{i}, x_{i}] = k\, \sum_{j \neq i}\sum_{l=0}^{m-1}s_{ij}
\alpha^{l}_{i} \alpha^{-l}_{j} + 
\sum_{l=1}^{m-1}c_{l}\alpha^{l}_{i}\ , \quad 1 \le i \le n \ , \\*[1ex]
&& [y_{i}, x_{j}] = - k\, \sum_{l=0}^{m-1}s_{ij} \epsilon^{l}
\alpha^{l}_{i} \alpha^{-l}_{j} , \quad 1 \le i \neq j \le n\ . \nonumber
\end{eqnarray} 
Write $ \bold{e} := \frac{1}{ | \bGamma_n | } \sum_{g \in \bGamma_n} g $
for the symmetrizing idempotent in $ \mathbb C\bGamma_n \subset H_{k,c} $
and denote by $ U_{k,c}=\bold{e} H_{k,c} \bold{e} \,$
the corresponding spherical subalgebra of $ H_{k,c} $.
Next, put $\,\sigma_{n}:=\frac{1}{n!}\sum_{\sigma \in S_{n}}\sigma \,$
and $\, \sigma_{n-1}:=\frac{1}{(n-1)!}\sum_{\sigma \in S_{n-1}}\sigma \,$,
where $S_{n-1}$ is the subgroup of $S_{n}$ permuting $ \{2,...,n \} $.
It is easy to see that $\, \bold{e}= \sigma_{n}(\ee_0
\otimes \ee_0 \otimes \ldots \otimes \ee_0) \,$, where
$\, \ee_0 = \frac{1}{|\Gamma|} \sum_{g \in \Gamma} g \,$. 
Now, recall the idempotents \eqref{idd} and, 
for each $\,i = 0,1,\ldots, m-1\,$, define
$$ 
\nu_{i}:=\sigma_{n-1}(\ee_{i} \otimes \ee_0 \otimes
\ldots \otimes \ee_0) \ . 
$$
One can check easily that $ \nu_i $ are idempotents in $ \C\bGamma_n $
satisfying the relations
$$
 x_{1} \cdot \nu_{i+1}=
\nu_{i} \cdot x_{1}\, , \quad  y_{1} \cdot \nu_{i} = \nu_{i+1}\cdot
y_{1}\ .
$$

The following lemma is a generalization of Lemma~\ref{Lee2}: it gives
an explicit homomorphism from $ B_\lambda $ to the spherical subalgebra of $ H_{k,c} $.
\begin{lemma}
\la{L7}
Let $ \Pi_\lambda = \Pi_{\lambda}(Q_\infty) $ be the deformed
preprojective algebra of weight $\, \lambda =
(\lambda_\infty,\, \lambda_0,\, \ldots\, \lambda_{m-1}) \in \C^{m+1} $
with $\, \lambda_\infty = - n \sum_{i=0}^{m-1} \lambda_i \,$, and let 
$ B_\lambda = e_\infty \Pi_{\lambda} e_\infty $. 
Then there is a natural homomorphism of unital algebras $\, 
\theta:\, B_\lambda \to U_{k,c}\,$, where $\, k = \frac{\lambda_\infty}{m n} \,$ and
$\, c =(\lambda_{0}+\frac{\lambda_{\infty}}{n})\ee_0 +
 \sum_{i=1}^{m-1}\lambda_{i}\ee_{i} \,$.
\end{lemma}
\begin{proof}
The assignment 
\begin{equation} 
X_{i} \, \mapsto \, \nu_{i}\,x_{1} \ , \quad
Y_{i} \, \mapsto \, - \, y_{1}\,\nu_{i}\ ,\quad  e_{i} \mapsto \nu_i \
, \nonumber
\end{equation}
\begin{equation}
v \mapsto \frac{1}{n} \lambda_{\infty}\, (1+s_{12}+\ldots+s_{1n})\, ,\quad
 w \mapsto \nu_{0}\ , \quad e_{\infty} \mapsto \bold{e} \ , \nonumber
\end{equation}
extends by multiplicativity to a linear
map $\, \theta:\, \mathbb C \bar{Q}_\infty \to H_{k,c}\,$.
While this map is {\it not} a homomorphism of unital rings 
(e.g., $ 0 = \theta(e_{\infty}\, e_{0}) \not= \theta(e_{\infty})\,\theta(e_{0})= \bold{e}$),
it becomes so when restricted to the subalgebra 
$\,e_{\infty} (\C \bar{Q}_\infty)  e_{\infty}\,$. Under this restriction, 
the image of $ \theta $ gets in $ U_{k,c} $. Thus, to prove the
theorem it suffices to check the defining relations
\eqref{relP11}--\eqref{relP12}, and this boils down to a routine calculation.
\end{proof}

\begin{remark} It appears that $ \theta $ is a special case of the map
$ \Theta^{\footnotesize \rm quiver} $ defined (somewhat implicitly) in
\cite{EGGO}, see (1.6.3). We thank Vitya Ginzburg for this remark.
\end{remark}

\vspace{1.5ex}

Write $ \theta_* : \Mod(U_{k,c}) \to \Mod(B_{\lambda}) $ for the
restriction functor corresponding to $ \theta $. Then our next result 
can be formulated as follows (cf. Theorem~\ref{T2}).
\begin{theorem}
\la{TT2}
Let $\, \tau \in \C^m $ be a regular vector for the Dynkin quiver \eqref{D44}, 
let $\, \alpha = (1,\,\bn)\in \Z^{m+1} $ with
$\, \bn=(n,n,\ldots,n)\,$, and let $\, \lambda := (-\tau\cdot \bn,\, \tau)\,$. 
Then the composition of functors
$$
\Xi_{\tau} :\, \Mod(H_{k,c}) \rTo^{\bold{e}} \Mod(U_{k,c})
\rTo^{\theta_{\ast}} \Mod(B_{\lambda}) \rTo^{j_{!\ast}}
\Mod(\Pi_{\lambda}) 
$$
maps the set $ \Irr(H_{k,c}) $ of isomorphism classes of simple $ H_{k,c}$-modules
bijectively onto the set $ \Irr(\Pi_\lambda, \, \alpha) $
of isomorphism classes of simple $\Pi_{\lambda}$-modules of dimension
vector $ \alpha $. This map agrees with Etingof-Ginzburg's construction in \cite{EG}.
\end{theorem}

As we pointed out earlier, there is a natural bijection between
$\, \bigsqcup_{n \ge 0} \Irr(\Pi_\lambda, \, \alpha) \,$ and  
the set $ \RR^{o}_{\tau} $ of isomorphism classes of stably free ideals of $O_{\tau}$  
(see Remark after Theorem \ref{TT1}). Thus, 
combining Theorems \ref{TT1} and \ref{TT2} together, we get
\begin{theorem}
\la{TT3}
The composition of functors
$$ 
\Mod(H_{k,c}) \rTo^{\bold{e}} \Mod(U_{k,c}) \rTo^{\theta_{\ast}} 
\Mod(B_\lambda) \rTo{j_{!}} \Mod(\Pi_\lambda) \rTo{i^{!}} 
\Mod(\S_\tau) \rTo{\ee} \Mod(O_{\tau}) 
$$
maps the set $ \Irr(H_{k,c}) $ of isomorphism classes of simple $H_{k,c}$-modules
injectively into the set of isomorphism classes of stably free 
ideals of $ O_{\tau} $. Collecting such maps for all $\,n \ge 0 \,$, 
we get a bijective correspondence $\, \bigsqcup_{n \ge 0}\Irr[H_{k,c}(\bGamma_{n})] 
\stackrel{\sim}{\to} \RR^{o}_{\tau} \,$.
\end{theorem}

\bibliographystyle{amsalpha}

\end{document}